\documentclass[hidelinks,12pt]{article}

\usepackage{amsfonts,amsmath, amssymb}

\usepackage{xcolor}
\usepackage{hyperref}
\usepackage{enumerate}
\usepackage{epsf,epsfig,amsfonts,a4wide}
\usepackage{amsfonts, mathdots}
\usepackage{amsmath,amssymb,amsthm}
\usepackage{url}
\usepackage{amsmath,amssymb,amsthm}
\usepackage{amssymb}
\usepackage{amsthm}
\usepackage{epsf,epsfig,amsfonts,graphicx}

\newtheorem{Theorem}{Theorem}[section]

\newtheorem{Lemma}{Lemma}[section]

\newtheorem{Corollary}{Corollary}[section]

\newtheorem{Remark}{Remark}[section]

\theoremstyle{remark}

\newcommand{\be}{\begin{equation}}
\newcommand{\ee}{\end{equation}}

\newcommand{\Id}{\textrm{\rm Id}}

\newcommand{\gl}{\mathrm{gl}}

\newcommand{\pd}[2]{\frac{\partial#1}{\partial#2}}

\newcommand{\dd}{{\mathrm d}\,}

\providecommand{\keywords}[1]
{
  \small	
  \textbf{\text{Keywords:}} #1
}

\newcommand{\footremember}[2]{%
    \footnote{#2}
    \newcounter{#1}
    \setcounter{#1}{\value{footnote}}%
}

\newcommand{\weg}[1]{}

\title{Nijenhuis operators with a unity and $F$-manifolds}
\author{%
  Evgenii I. Antonov\footremember{alley}{Institut f\"ur Mathematik, Friedrich Schiller Universit\"at Jena, 07737 Jena Germany, \\
  \href{evgenii.antonov@uni-jena.de}{\textbf{evgenii.antonov@uni-jena.de}}}%
   \and \& \and Andrey Yu. Konyaev\footremember{trailer}{Faculty of Mechanics and Mathematics, Moscow State University, and Moscow Center for Fundamental and
Applied Mathematics, 119992, Moscow Russia, 
\href{maodzund@yandex.ru}{\textbf{maodzund@yandex.ru}}}%
}
\date{}

\begin{document}

\maketitle

\begin{abstract}
    The core object of this paper is a pair $(L, e)$, where $L$ is a Nijenhuis operator and $e$ is a vector field satisfying a specific Lie derivative condition, i.e., $\mathcal{L}_{e}L=\operatorname{Id}$.
    
    Our research unfolds in two parts. In the first part, we establish a Splitting Theorem for Nijenhuis operators with a unity, offering an effective reduction of their study to cases where $L$ has either one real or two complex conjugate eigenvalues at a given point. We further provide the normal forms for $\gl$-regular Nijenhuis operators with a unity around algebraically generic points, along with semi-normal forms for dimensions two and three. 

    In the second part, we establish the relationship between Nijenhuis operators with a unity and $F$-manifolds. Specifically, we prove that the class of regular $F$-manifolds coincides with the class of Nijenhuis manifolds with a cyclic unity. By extending our results from dimension three, we reveal semi-normal forms for corresponding $F$-manifolds around singularities. 

\end{abstract}



\section{Introduction}

Nijenhuis operators first appeared in work by A.Nijenhuis \cite{nij} in 1951 in the study of the integrability problem for eigendistributions of operator fields. Later the results of the paper led to the development of several differential concomitants, Fr\"olicher-Nijenhuis bracket in particular (see \cite{ks} for a historical overview on the subject). In following years Nijenhuis torsion and Nijenhuis operators appeared in different branches of mathematics and mathematical physics. Perhaps, the most famous example is the Newlander-Nirenberg theorem.

In late 2010-s A. Bolsinov, A. Konyaev, and V. Matveev initiated the program to study Nijenhuis operators systematically \cite{bkm1}. The idea was to look at Nijenhuis operators, which are usually treated as secondary objects, as primary ones. This approach proved to be quite fruitful \cite{bkm2, bkm3, bkm4, bkm5, bkm6, k}. Typically, Nijenhuis operators come with a "companion" object --- a Poisson bracket \cite{bkm3}, a metric \cite{bkm4, bkm5}, a vector field \cite{bkm6} --- with some condition of compatibility.

The pair $\left(L, e\right)$ on a manifold $\mathsf M^n$, where $L$ is an operator field and $e$ is a vector field, is called  {\it a Nijenhuis operator with a unity} if $L$ is a Nijenhuis and $e$ satisfies the identity
\begin{equation}\label{unitycond}
    \mathcal{L}_{e}L = \Id.
\end{equation}
Here $\mathcal L_e$ stands for Lie derivative along a vector field $e$. Taking $\operatorname{trace}$ of the both sides of equality \eqref{unitycond}, we obtain:
\begin{equation}\label{tracecond}
e(\operatorname{tr}L) = n.  
\end{equation}
It implies that $e \neq 0$ everywhere on $\mathsf M^n$. The main object of our study is the pair $(L,e)$. 

If the category (analytic or smooth) is not stated, the results hold in both. In this work, we show:
\begin{itemize}
    \item We prove the Splitting Theorem for Nijenhuis operators with a unity, effectively reducing the study of such structures to the case when $L$ has either one real or two complex conjugate eigenvalues at a given point. We stress that splitting works in both regular and singular points.
    \item We provide the local normal form of $\gl$-regular Nijenhuis operators with a unity when $L$ has either one real or two conjugate complex eigenvalues in a given algebraically generic point.
    \item We provide local semi-normal forms for all Nijenhuis operators with a unity in dimension two.
    \item We provide a local semi-normal form for Nijenhuis operators with a unity in dimension three in a neighborhood of the following singular point $\mathsf p_0$: $L$ has two distinct eigenvalues at a point of general position and the linear operator $L(\mathsf p_0)$ has a single eigenvalue.
\end{itemize}
In the second part of our study, we explore the relationship between Nijenhuis operators with a unity and $F$-manifolds. We prove that the class of regular $F$-manifolds coincides with the class of Nijenhuis manifolds with a cyclic unity.

Furthermore, we apply the results in dimension three to provide semi-normal forms of corresponding $F$-manifolds around singularities. This result is new for the theory of $F$-manifolds. In particular, we construct a lot of new examples in this special case.


\section{General theory of Nijenhuis operators with a unity}

Let $\mathsf M^n$ be a smooth or analytic manifold of dimension $n$. For any operator field $L$ {\it the Nijenhuis torsion of $L$} is defined as
$$
 \mathcal N_L(\xi, \eta) = [L\xi, L\eta] + L^2[\xi, \eta] - L[L\xi, \eta] - L[\xi, L\eta].
$$
Here $\xi, \eta$ are arbitrary vector fields and the square brackets stand for the standard commutator of vector fields. The r.h.s. of this equation defines a tensor field of type $(1, 2)$, skew-symmetric in lower indices. We say that $L$ is {\it Nijenhuis operator} if its Nijenhuis torsion vanishes.

Since a Nijenhuis torsion $\mathcal{N}_{L}$ is a tensor of type $(1,2)$ we can also interpret it as a linear map from $T\mathsf M^n$ to $\operatorname{End}(T \mathsf M^n)$:
\begin{equation}
\label{nijendef2}
\mathcal{N}_{L}: \xi \mapsto L \mathcal{L}_{\xi} L-\mathcal{L}_{L \xi} L.
\end{equation}
See \cite{bkm1} for details and other definitions of Nijenhuis operators.

A point $ \mathsf p \in \mathsf M^n$ is called {\it algebraically generic for $L$} if the Segre characteristic of $L$ does not change in some neighborhood $U(\mathsf p) \subset \mathsf M^n$, otherwise it is called {\it singular}.

The following examples show the basic properties of Nijenhuis operators with a unity.
\begin{itemize}
    \item Let $L$ be a Nijenhuis operator conjugated to a standard Jordan block in a neighborhood of a point $\mathsf p_0 \in \mathsf M^n$ and its eigenvalue $\lambda_0$ is constant. Then it can be brought to the following form:
     \begin{equation*}
        L =  \left(\begin{array}{cccccc}
 \lambda_0 & & & & & \\
 1 &  \lambda_0 & & & & \\
 & 1 & \ddots & & & \\
 &  & \ddots  & \ddots & & \\
 & &  & 1 &  \lambda_0 &\\
 &  &   &  &  1 &  \lambda_0
\end{array}\right).
\end{equation*}
 The equality \eqref{tracecond} is an obstruction to the existence of a unity vector field $e$ for $L$.   
 \item 
Let $L$ be a differentially non-degenerate Nijenhuis operator. Then in canonical coordinates $\left(u^1, \dots, u^n\right)$
$$
L = \left( \begin{array}{ccccc}
     u^1 & 1 & 0 & \dots & 0  \\
     u^2 & 0 & 1 & \dots & 0  \\
      & & & \ddots &  \\
     u^{n - 1} & 0 & 0 & \dots & 1  \\
     u^n & 0 & 0 & \dots & 0  \\
\end{array}\right).
$$
The coordinate origin is a singular point: at this point, $L$ is a Jordan block of maximal size with zero eigenvalues, while in every neighborhood of coordinate origin, there exist points at which the eigenvalues of $L$ are real and pairwise distinct. By direct computation the vector field
$$
e = \left(n, - (n - 1)u^1, - (n - 2) u^2, \dots, - 2 u^{n - 2}, - u^{n - 1}\right)^T
$$
is the unity vector field for $L$.

\end{itemize}

The next Theorem is the first result of this paper.
\begin{Theorem}\label{t1}
Let $L$ be a  Nijenhuis operator with a unity $e$. Assume that at a point $\mathsf p$, its characteristic polynomial  $\chi_L(\lambda)=\det(\lambda\,\Id - L(\mathsf p))$ is factorised as $\chi_L(\lambda) = \chi_1(\lambda)\chi_2(\lambda)$, where $\chi_1(\lambda)$ and $\chi_2(\lambda)$ are coprime monic polynomials. Then in a~neighbourhood of $\mathsf p$ there exists a coordinate system 
$$
\underbrace{u^1_1, \dots, u^{m_1}_1}_{u_1}, \underbrace{u^1_2, \dots, u^{m_2}_2}_{u_2},  \quad \text{where} \quad m_1 = \operatorname{deg} \chi_1(\lambda),  m_2 = \operatorname{deg} \chi_2(\lambda)
$$
such that
\begin{enumerate}
    \item The Nijenhuis operator $L$ has the form
    $$
    L(u_1, u_2) = \begin{pmatrix}
         L_1(u_1) & 0  \\
         0 & L_2(u_2) 
     \end{pmatrix},
    $$
    where each of $L_i$ is Nijenhuis,  $i = 1, 2$. Moreover, $\chi_{L_1}(\lambda) = \chi_1(\lambda)$ and $\chi_{L_2}(\lambda) = \chi_2(\lambda)$.
    \item The unity vector field $e$ is decomposed into the sum $e = e_1 + e_2$, where 
    $$
    e_i = e_\alpha^i \pd{}{u^\alpha_i}, \quad \alpha = 1, \dots, m_i,
    $$
    and $e^\alpha_i$ depends only on variables $u_i$.
    \item Each summand $e_i$ is a unity vector field for respective $L_i$.
\end{enumerate}
\end{Theorem}
The first statement of Theorem \ref{t1} is the splitting theorem for Nijenhuis operators, see \cite[Theorem 3.1]{bkm1}, \cite[Theorems 1 and 2]{splitting}. Theorem \ref{t1} reduces the description of a Nijenhuis operator with a unity in a neighborhood of a given point to a description of a Nijenhuis operator with a unity with the additional condition: operator $L$ at $\mathsf p$ has only one real or two complex conjugate eigenvalues. We stress that in the statement of Theorem \ref{t1} the Segre characteristic of $L_i$ might vary from point to point, namely $\mathsf p$ is not algebraically generic.

There is a natural class of Nijenhuis operators called $\gl$-regular. A linear operator $L$ is {\it $\gl$-regular} if one of the following equivalent conditions hold:
\begin{itemize}
\item For each eigenvalue of $L$ there is exactly one Jordan block in its Jordan normal form (this includes complex eigenvalues).
\item There exists a vector $\xi$, such that $\xi, L\xi, \dots, L^{n - 1}\xi$ are linearly independent. Such vector is called {\it cyclic}.
\item The operators $\Id, L, \ldots, L^{n-1}$ are linearly independent.
\end{itemize}
We say that a Nijenhuis operator $L$ is $\gl$-regular if it is $\gl$-regular at every point. In particular, this means that the Segre characteristic of $L$ might vary, but in a very restricted way: an operator always stays $\gl$-regular.

\begin{Theorem}\label{t2}
Let $L$ be a $\gl$-regular Nijenhuis operator with a unity $e$ on a manifold $\mathsf M^n$. Assume that $\mathsf p \in \mathsf M^n$ is an algebraically generic point and has either one real or two complex conjugate eigenvalues. Then
\begin{itemize}
    \item If $L(\mathsf p)$ has a single real eigenvalue $\lambda_0$, then in a neighbourhood of $\mathsf p$ there exists a unique local coordinate system $\left(u^1, \dots, u^{n}\right)$ such that 
    \begin{equation}
        L =  \left(\begin{array}{cccccc}
u^1 + \lambda_0 & & & & & \\
 1 & u^1 + \lambda_0 & & & & \\
-u^3 & 1 & \ddots & & & \\
\vdots &  & \ddots  & \ddots & & \\
-(n-3)u^{n-1} & &  & 1 & u^1 + \lambda_0 &\\
(n-2)u^{n} &  &   &  &  1 & u^1 + \lambda_0
\end{array}\right),  \quad e = \pd{}{u^1}. 
\end{equation}
\item if $L(\mathsf p)$ has two complex conjugate eigenvalues $\mu_0,\Bar{\mu}_0$, where $\mu_0 = a_0 +i b_0$, then in a neighbourhood of $\mathsf p$ there exists a unique local coordinate system $\left(x^1,y^1, \dots, x^s, y^s \right), 2s = n$ such that 
\begin{equation}
    L = \left(\begin{array}{cccccc}
C^1 + \Lambda_0 & & & & & \\
 I & C^1 + \Lambda_0 & & & & \\
-C^3& I & \ddots & & & \\
\vdots &  & \ddots & \ddots & & \\
-(s-3)C^{s-1} & &  & I & C^1 + \Lambda_0 &\\
-(s-2)C^{s} &  &  &  &  I & C^1 + \Lambda_0
\end{array}\right), \quad e_i = \pd{}{x^1},
\end{equation}
where 
\begin{align*}
    C^p = &\left(\begin{array}{cc}
x^p & - y^p  \\
y^p & x^p  
\end{array}\right) \quad \text{for} \quad p \in \{1,\dots, s\}, \\
    \Lambda_{0} = &\left(\begin{array}{cc}
a_0 & - b_0  \\
b_0 & a_0  
\end{array}\right) , \quad I = \left(\begin{array}{cc}
1 & 0  \\
0 & 1  
\end{array}\right).
\end{align*}
\end{itemize}
\end{Theorem}

The normal forms of $L$, used in the statement of Theorem \ref{t2}, were introduced in \cite{bkm1}. The solution of \eqref{unitycond} for these normal forms is not unique: there are a lot of vector fields, such that $\mathcal L_e L = \operatorname{Id}$. At the same time, the normal forms possess a symmetry group. Theorem \ref{t2} implies that this symmetry group acts transitively on all vector fields, satisfying \eqref{unitycond}, allowing one to effectively straighten the field. This simple observation allows one to formulate a version of Theorem \ref{t2} for different normal forms of Nijenhuis operators, the so-called upper triangular Toeplitz form.

\begin{Theorem}\label{t2_2}
Let $L$ be a $\gl$-regular Nijenhuis operator with a unity $e$ on a manifold $\mathsf M^n$. Assume that $\mathsf p \in \mathsf M^n$ is an algebraically generic point and has either one real or two conjugate complex eigenvalues. Then
\begin{itemize}
    \item if $L(\mathsf p)$ has single real eigenvalue $\lambda_0$, then in a neighbourhood of $\mathsf p$ there exists a unique local coordinate system $\left(u^1, \dots, u^{n}\right)$ such that 
    \begin{equation}
        L =  \left(\begin{array}{cccccc}
u^1 + \lambda_0 & & & & & \\
u^2+ 1 & u^1 + \lambda_0 & & & & \\
u^3 & \ddots & \ddots & & & \\
\vdots &  \ddots & \ddots & \ddots & & \\
u^{n-1} & & \ddots & \ddots & u^1 + \lambda_0 &\\
u^{n} & u^{n-1} &  \hdots & u^3 & u^2 + 1 & u^1 + \lambda_0
\end{array}\right),  \quad e = \pd{}{u^1}. 
\end{equation}
\item if $L(\mathsf p)$ has two complex conjugate eigenvalues $\mu_0,\Bar{\mu}_0$, where $\mu_0 = a_0 +i b_0$, then in a neighborhood of $\mathsf p$ there exists a unique local coordinate system $\left(x^1,y^1, \dots, x^s, y^s \right), 2s = n$ such that 
\begin{equation}
    L = \left(\begin{array}{cccccc}
C^1 + \Lambda_0 & & & & & \\
C^2+ I & C^1 + \Lambda_0 & & & & \\
C^3& \ddots & \ddots & & & \\
\vdots &  \ddots & \ddots & \ddots & & \\
C^{s-1} & & \ddots & \ddots & C^1 + \Lambda_0 &\\
C^{s} & C^{s-1} &  \hdots & C^3 & C^2 + I & C^1 + \Lambda_0
\end{array}\right), \quad e_i = \pd{}{x^1},
\end{equation}
where 
\begin{align*}
    C^p = &\left(\begin{array}{cc}
x^p & - y^p  \\
y^p & x^p  
\end{array}\right) \quad \text{for} \quad p \in \{1,\dots, s\}, \\
    \Lambda_{0} = &\left(\begin{array}{cc}
a_0 & - b_0  \\
b_0 & a_0  
\end{array}\right) , \quad I = \left(\begin{array}{cc}
1 & 0  \\
0 & 1  
\end{array}\right).
\end{align*}
\end{itemize}
\end{Theorem}

We say that $L$ is brought to a {\it semi-normal form} if the functional parameters of the corresponding form are defined up to equivalence. The equivalence relation on functional parameters is defined by a coordinate change of an operator. The next Theorem provides all semi-normal forms of Nijenhuis operators with a unity in dimension two in the analytic category.

\begin{Theorem}\label{t3}
Let $L$ be a Nijenhuis operator with a unity $e$ in the analytic category in dimension two. Then in a neighbourhood of a point $\mathsf p$ the pair $\left(L, e\right)$ can be brought to one of the following forms:
\begin{enumerate}
    \item $$e = \partial_x, \quad L = \left( \begin{array}{cc}
         \lambda_0 + x & 0  \\
         0 & \lambda_0 + x
    \end{array}\right), \quad \lambda_0 \in \mathbb{R}.$$
    \item $$e = \partial_x, \quad L = \left( \begin{array}{cc}
         \lambda_0 & 0  \\
         0 & \lambda_0
    \end{array}\right) + \left( \begin{array}{cc}
         x & - \frac{1}{2}  \\
         2(y+d) & x
    \end{array}\right), \quad d, \lambda_0 \in \mathbb{R}.$$
    \item $$e = \partial_x, \quad L = \left( \begin{array}{cc}
         \lambda_0 & 0  \\
         0 & \lambda_0
    \end{array}\right) + \left( \begin{array}{cc}
         x &  \mp \frac{k}{2} y^{k-1}  \\
        \frac{2}{k} y  & x
    \end{array}\right), \quad k \in \mathbb{N}, \lambda_0 \in \mathbb{R}.$$
    
    \item $$e = \partial_x, \quad L = \left( \begin{array}{cc}
         \lambda_0 & 0  \\
         0 & \lambda_0
    \end{array}\right) + \left( \begin{array}{cc}
         x & 0  \\
         f(x, y) & x
    \end{array}\right), \quad \lambda_0 \in \mathbb{R}.$$
    Here $f(x, y)$ is an arbitrary function. 

    Two semi-normal forms with functional parameters $f$ and $\bar f$ respectively are equivalent if and only if there exists a function $h(x, y)$, such that $\bar f(x, h(x, y)) = \pd{h}{y}(x, y) f(x, y)$ and $\pd{h}{y} (0, 0) \neq 0$. 
\end{enumerate}
\end{Theorem}

The next Theorem treats a special case of Nijenhuis operators with a unity. As we will see later, it is useful from the point of view of applications. The other cases will be treated elsewhere.

\begin{Theorem}\label{t4}
Let $L$ be a Nijenhuis operator with a unity $e$ in dimension three in the analytic category. Assume that $L$ has two distinct eigenvalues at a point of general position. Then in a neighbourhood of a point $\mathsf p_0 \in \mathsf M^n$, such that $L(\mathsf p_0)$ has a single eigenvalue, the pair $\left(L, e\right)$ can be brought to the following form:
 \begin{equation*}
   \label{prenormalform}
  e = \pd{}{x^1}, \quad L=\left(\begin{array}{lll}
x^{1} +\lambda_0 & 0 & 0 \\
\frac{x^{2}}{k} &  \pm (x^{2})^{k} + x^1 +\lambda_0 & 0 \\
f(x^2,x^3) & g(x^2,x^3) & \pm (x^{2})^{k} + x^1 +\lambda_0 \\
\end{array}\right), \quad k \in \mathbb{N}, \lambda_0 \in \mathbb{R}.
\end{equation*}
 Here $f$ and $g$ are analytic functions and the following PDE is satisfied:
\begin{equation*}
 \frac{x^{2}}{k}\pd{g}{x^2} + f\pd{g}{x^3} - g\pd{f}{x^3} = \frac{k-1}{k} g. 
\end{equation*}
Two semi-normal forms with functional parameters $f, g$ and $\Bar{f}, \Bar{g}$ respectively are equivalent if and only if there exists a function $h(x^2,x^3)$, $\pd{h}{x^3}(0,0) \neq 0$ such that the following identities hold:
\begin{align*}
& \Bar{g}(x^2, h(x^2,x^3))=g\left(x^2, x^3\right) \frac{\partial h}{\partial x^3}\left(x^2, x^3\right), \\
& \Bar{f}(x^2, h(x^2,x^3))=\frac{x^2}{k} \frac{\partial h}{\partial x^2}\left(x^2, x^3\right)+f\left(x^2, x^3\right) \frac{\partial h}{\partial x^3}\left(x^2, x^3\right) . 
\end{align*}

\end{Theorem}

\begin{Corollary}\label{c1}
In the assumptions of Theorem \ref{t4} suppose that in the semi-normal form $f(0, 0) \neq 0$. Then there exists a coordinate system, in which
\begin{equation*}
    e = \pd{}{x^1}, \quad L = \left(\begin{array}{ccc}
    x^1 +\lambda_0 & 0 & 0 \\
    \frac{x^{2}}{k} & \pm (x^{2})^{k} + x^1 +\lambda_0 & 0  \\
     1 &  F(x^{2}e^{-\frac{x^{3}}{k}})e^{\frac{(k-1)x^{3}}{k}} & \pm (x^{2})^{k} +x^1 +\lambda_0
  \end{array}\right), \quad k \in \mathbb{N}, \lambda_0 \in \mathbb{R},
\end{equation*}
where $F$ is an analytic function.
Two semi-normal forms with functional parameters $F$ and $\Bar{F}$ respectively are equivalent if and only if there exists a function $q(x^2)$ such that $\Bar{F}(x^{2}e^{-\frac{x^{3} + x^{2}q(x^2)}{k}})e^{\frac{(k-1)x^{2}q(x^2)}{k}}= F(x^{2}e^{-\frac{x^{3}}{k}})$. 
\end{Corollary}

\begin{Corollary}\label{c2}
In the assumptions of Theorem \ref{t4} suppose that in the semi-normal form $g(0, 0) \neq 0$. Then there exists a coordinate system, in which
\begin{equation*}
    e = \pd{}{x^1}, \quad L = \left(\begin{array}{ccc}
    x^1 +\lambda_0 & 0 & 0 \\
    \frac{x^{2}}{k} & \pm (x^{2})^{k} + x^1 +\lambda_0 & 0  \\
     \frac{1-k}{k}x^{3}  & 1 & \pm (x^{2})^{k} +x^1 +\lambda_0
  \end{array}\right), \quad k \in \mathbb{N}, \lambda_0 \in \mathbb{R}.
\end{equation*}
\end{Corollary}

The formulas in Theorem \ref{t4} imply that $(x^2/k, f(x^2, x^3))^T$ is transformed under the coordinate change $\bar x^2 = x^2, \bar x^3 = h(x^2, x^3)$ as a vector field. There are two distinct cases: $f(0, 0) \neq 0$ and $f(0, 0) = 0$. The first is treated by Corollary \ref{c1}. The second is related to the problem of linearization of vector fields on a plane, which is a classical and very complicated problem.


\section{Applications to \texorpdfstring{$F$}{F}-manifolds}

To define an $F$-manifold one starts with a manifold $\mathsf M^n$, equipped with a triple: a vector field $e$, a vector field $E$ and a tensor of type $(1, 2)$ denoted by $a$. The tensor $a$ induces a natural bilinear operation on vector fields $\xi, \eta$ as
$$
\xi \circ \eta = a(\xi, \eta)
$$
The definition of an $F$-manifold is given in terms of $e, E$ and an operation $\circ$: we say that $\left(\mathsf M^n, \circ, e, E\right)$ is an {\it $F$-manifold} if $\circ, e, E$ satisfy the following conditions:
\begin{enumerate}
    \item $\circ$ defines the structure of commutative associative algebra on vector fields,
    \item $e$ is a unity of the algebra $\circ$,
    \item For any vector fields $\xi, \eta, \zeta, \theta$ we have
    \begin{equation}\label{hertlingmanincond}
    \begin{aligned}
    0 & = [\xi \circ \eta, \zeta \circ \theta] - [\zeta, \xi \circ \eta] \circ \theta - \zeta \circ [\xi \circ \eta, \theta] - \xi \circ [\eta, \zeta \circ \theta] + \xi \circ [\eta, \zeta] \circ \theta + \\
    & + \xi \circ \zeta \circ [\eta, \theta] - \eta \circ [\xi, \zeta \circ \theta] + \eta \circ [\xi, \zeta] \circ \theta + \eta \circ \zeta \circ [\xi, \theta].
    \end{aligned}    
    \end{equation}
    Here the square brackets $\left[\cdot, \cdot \right]$ define the standard commutator of vector fields. 
    \item For any pair of vector fields $\xi, \eta$
    \begin{equation}\label{eulervectorfieldcond}
        [E, \xi \circ \eta] - [E, \xi]\circ \eta - \xi \circ [E, \eta] = \xi \circ \eta.
    \end{equation}
    A vector field $E$ is called an {\it Euler vector field}.
\end{enumerate}
It is not obvious, but the formula \eqref{hertlingmanincond} defines the tensor field of type $(1, 4)$. The formula itself is related to the so-called Ako-Yano bracket (see \cite{aky} and discussion in \cite{magri}), which plays an important role in the theory of quasilinear integrable systems. The condition \eqref{eulervectorfieldcond} can be rewritten as $\mathcal L_E a = a$, where $\mathcal L_E$ stands for the Lie derivative along a vector field $E$.

The concept of an $F$-manifold was introduced by Hertling and Manin in \cite{hm, h1} as a generalization of Frobenius manifolds. $F$-manifolds found an application in many areas of mathematics and mathematical physics (see \cite{h2} for overview and further references and \cite{dav1, lpr, l}). In \cite{arsie} (Theorem 4.4) Arsie and Lorenzoni have shown  that operator field $L$, defined by the identity
$$
L \xi = E \circ \xi,
$$
is, in fact, a Nijenhuis operator. Thus, establishing that $F$-manifolds are Nijenhuis manifolds in general.

In \cite{dav2} David and Hertling introduced a large class of {\it regular} $F$-manifolds, i.e., $\left(\mathsf M^n, \circ, e, E\right)$ such that a Nijenhuis operator $L:=E\circ$ is $\gl$-regular.

A triple $(\mathsf M^n, L, e)$ is called \textit{Nijenhuis manifold with a unity} if $\mathsf M^n$ is a manifold, and $(L,e)$ is a Nijenhuis operator with a unity. The next theorem is the main result of the paper.

\begin{Theorem}\label{t5}
The class of regular $F$-manifolds coincides with the class of Nijenhuis manifolds with a cyclic unity.
\end{Theorem}
\begin{Remark}
   The cyclicity of a unity guarantees that the corresponding Nijenhuis operator is $\gl$-regular.
   \end{Remark}

Using Theorem \ref{t5} and Theorem \ref{t4} one can obtain a semi-normal form for $F$-manifolds with certain algebraic constraints.

\begin{Theorem}\label{t6}
Let $\left(\mathsf M^3, e, E, \circ\right)$ be a three-dimensional $F$-manifold in the analytic category such that at a generic point, $L$ is $\gl$-regular and has two distinct eigenvalues. Then in a neighborhood of a point $\mathsf p \in \mathsf M^3$ such that $L(\mathsf p_0)$ has a single eigenvalue, there exist local coordinates $\left(x^1, x^2, x^3\right)$ in which
\begin{enumerate}
\item $e = \pd{}{x^1}.$
\item $E = (x^1 + \lambda_0) \pd{}{x^1} + \frac{1}{k} x^2 \pd{}{x^2} + f(x^2, x^3) \pd{}{x^3}.$
\item The operation $\circ$ is defined as 
$$
\begin{aligned}
    & \pd{}{x^1} \circ \pd{}{x^i} = \pd{}{x^i}, \quad i = 1, 2, 3, \\
    & \pd{}{x^2} \circ \pd{}{x^2} = \pm k(x^2)^{k - 1} \pd{}{x^2} + h(x^2, x^3) \pd{}{x^3}, \\
    & \pd{}{x^2} \circ \pd{}{x^3} = \pm k(x^2)^{k - 1} \pd{}{x^3}, \\
    & \pd{}{x^3} \circ \pd{}{x^3} = 0.
\end{aligned}  
$$
where $f$ and $h$ are analytic functions and the following PDE is fulfilled:
\begin{equation*}
    \frac{x^2}{k}\pd{h}{x^2} + f\pd{h}{x^3} -k(x^2)^{k-1}\pd{f}{x^2} - h\pd{f}{x^3} = \frac{k-2}{k}h.
\end{equation*}
Two semi-normal forms with functional parameters $f,h$ and $\Bar{f}, \Bar{h}$ respectively are  equivalent if and only if there exists a function $r(x^2,x^3), \pd{r}{x^3}(0,0) \neq 0$ such that the following identities hold:
\begin{align*}
& \Bar{h}(x^2, r(x^2,x^3))=k(x^2)^{k-1}\pd{r}{x^2}\left(x^2,x^3\right)   +  h\left(x^2, x^3\right) \frac{\partial r}{\partial x^3}\left(x^2, x^3\right), \\
& \Bar{f}(x^2, r(x^2,x^3))=\frac{x^2}{k} \frac{\partial r}{\partial x^2}\left(x^2, x^3\right)+f\left(x^2, x^3\right) \frac{\partial r}{\partial x^3}\left(x^2, x^3\right). 
\end{align*}
\end{enumerate}
\end{Theorem}
\section{Proof of Theorem \ref{t1}}
The factorization of the characteristic polynomial $\chi_{L(\mathsf p)}(t) = \chi_1(t)\chi_2(t)$ at a point $\mathsf p$ can be extended by continuity to a certain neighborhood $U(\mathsf p)$. Then a coordinate system $(u^1,u^2)$ given in the theorem is adopted to the decomposition of $TU = \mathcal{D}_1 \oplus \mathcal{D}_2$ into two integrable distributions $\mathcal{D}_1 := \operatorname{Ker}\chi_1(L), \mathcal{D}_2 := \operatorname{Ker}\chi_2(L)$.

The first statement of Theorem \ref{t1} is the splitting theorem for Nijenhuis operators (see \cite[Theorem 3.1]{bkm1} and also \cite[Theorems 1 and 2]{splitting}). 

We are left to prove the second and the third statements of the theorem.  As a unity vector field is non-vanishing, we may choose a coordinate system $(x^1,\dots, x^n)$ such that $e = \pd{}{x^1}$ and $L$ is as follows:
\begin{equation*}
    L = A(x^2,\dots,x^n) + x^1 \operatorname{Id}.
\end{equation*}
Then we have:
\begin{equation*}
  \chi_{L}(t) = \operatorname{det}(L - t\operatorname{Id}) = \operatorname{det}(A - (t - x^1)\operatorname{Id}) = \chi_{A}(t-x^1), 
\end{equation*}
and the factorization of the characteristic polynomial $\chi_{L} = \chi_1(t)\chi_2(t)$ can be written as:
\begin{equation*}
  \chi_{L}(t) = p_{1}(t-x^1)p_{2}(t-x^1),
\end{equation*}
where coefficients of $p_1$ and $p_2$ depend only on variables $(x^2,\dots,x^n)$ and $p_{1}(t-x^1) = \chi_1(t),p_{2}(t-x^1) = \chi_2(t)$.

We claim that the distributions $\mathcal{D}_1, \mathcal{D}_2$ are preserved by $e = \pd{}{x^1}$, i.e., if vector fields $\eta_i \in \mathcal{D}_i$ for $i=1,2$ then  $\left[\pd{}{x^1}, \eta_i \right] \in \mathcal{D}_i$.

Assume that $p_1(t) = t^k + \omega_1 t^{k-1} + \dots + \omega_k$, where $\omega_i$ are the functions depending on $(x^2,\dots, x^n)$. Then we have:
\begin{equation*}
 \chi_1(t) =  p_{1}(t-x^1) =   (t-x^1)^k + \omega_1  (t-x^1)^{k-1} + \dots + \omega_k,
\end{equation*}
and
\begin{equation*}
 \chi_1(L) =  p_{1}(L-x^1\operatorname{Id}) =   p_{1}(A)  = A^k + \omega_1  A^{k-1} + \dots + \omega_k \operatorname{Id}.
\end{equation*}
Finally
\begin{equation*}
  \chi_1(L)\left[\pd{}{x^1}, \eta_1 \right] = \mathcal{L}_{\pd{}{x^1}}(\chi_1(L) \eta_1) - \mathcal{L}_{\pd{}{x^1}}(\chi_1(L)) \eta_1.   
\end{equation*}
As $\eta_1 \in \mathcal{D}_1$, the first summand vanishes. The second summand vanishes because $\chi_1(L)$ is $x^1$ independent. The proof of $\left[\pd{}{x^1}, \eta_2 \right] \in \mathcal{D}_2$ goes in the same manner.

Thus, taking a coordinate system $(u_1,u_2)$ adopted to the decomposition of $TU = \mathcal{D}_1 \oplus \mathcal{D}_2$ we obtain that the unity vector field $e$ is decomposed into the sum $e = e_1 + e_2$, where 
    $$
    e_i = e_\alpha^i \pd{}{u^\alpha_i}, \quad \alpha = 1, \dots, m_i, i = 1, 2,
    $$
    and $e^\alpha_i$ depends only on variables $u_i$.

As a corollary, we get that the restriction of the identity $\mathcal{L}_{e}(L) = \operatorname{Id} $ to integral manifolds of $\mathcal{D}_1$ and $\mathcal{D}_2$  leads to the following identities:
\begin{align*}
  \mathcal{L}_{e_1}(L_1) &= \operatorname{Id}, \\
  \mathcal{L}_{e_2}(L_2) &= \operatorname{Id}. 
\end{align*}
Therefore, $e_1$ and $e_2$ are unity vector fields for $L_1$ and $L_2$ respectively. 
\begin{Corollary}
\label{charcoefsrelationscor}
    Let $\left(L,e\right)$ be a Nijenhuis operator with a unity on a manifold $\mathsf M^n$, and $\chi_{L}(t) =t^n +\sigma_1 t^{n-1}+ \dots + \sigma_n $ be the characteristic polynomial of $L$. Then the following differential conditions hold:
    \begin{equation}
    \label{charcoefsrelations}
        e(\sigma_k) = - (n-k+1)\sigma_{k-1}, \quad  1 \leq k \leq n, \quad \sigma_0 := 1.
    \end{equation}
    \end{Corollary}
   Firstly, consider a neighborhood of an algebraically generic point $\mathsf p$, where the Segre characteristic of $L$ does not change. By means of Theorem \ref{t1} we may think of $L$ as a Nijenhuis operator with a single eigenvalue $\lambda$. Then we have:
   \begin{equation*}
       e(\lambda) = \frac{1}{n}e(\operatorname{tr}L) =\frac{1}{n}\operatorname{tr}\mathcal{L}_{e}(L) = 1. 
   \end{equation*}
   Since $\sigma_k = (-1)^k\binom{n}{k}\lambda^k$, we obtain:
   \begin{equation*}
       e(\sigma_k) = k (-1)^k\binom{n}{k}\lambda^{k-1} = - (n-k+1)\sigma_{k-1}, \quad \text{for $1 \leq k \leq n$ }
   \end{equation*}
As the conditions \eqref{charcoefsrelations} hold for any algebraically generic point $\mathsf p \in \mathsf M^n$, they hold everywhere on $\mathsf M^n$ by continuity. 
\section{Proof of Theorems \ref{t2} and \ref{t2_2}}
\subsection{Case of a real eigenvalue \texorpdfstring{$\lambda$}{l}}
By assumption, we have that $\mathcal{L}_{e}(\lambda) = 1$, thus the eigenvalue of $L$ is non-constant. 
The following theorem (see \cite{bkm1}) states the local normal form for a Nijenhuis operator.
\begin{Theorem}
\label{thmjrdn}
 Suppose that in a neighborhood of a generic point $\mathsf p \in \mathsf M^n$ a Nijenhuis operator $L$ is similar to the standard Jordan block with a real eigenvalue $\lambda$, and $d \lambda(\mathsf p) \neq 0$. Then there exists a local coordinate system $\left(u^1, \ldots, u^n\right)$ in which $L$ takes the following form:
\begin{equation}
\label{jrdn}
 L(u)=L_{\text {can }}=
 \left(\begin{array}{cccccc}
u^1 + \lambda(\mathsf p) & & & & & \\
1 & u^1 + \lambda(\mathsf p)& & & & \\
-u^3 & 1 &  & & & \\
\vdots & &  & \ddots & & \\
-(n-3)u^{n-1} & & & 1 & u^1 + \lambda(\mathsf p) & \\
-(n-2)u^n & & & & 1 & u^1 + \lambda(\mathsf p)
\end{array}\right). 
\end{equation}
\end{Theorem}

Our goal now is to prove that it is possible to bring a unity $e$ to its normal form without changing $L$.

\begin{Remark}
  Let $\lambda: U(\mathsf p) \rightarrow \mathbb{R}$ be the eigenvalue of $L$ considered as a smooth function and denote $L_\lambda:=L-\lambda \operatorname{Id}$.
 The coordinate system in Theorem \ref{thmjrdn} is adapted to a flag of integrable distributions
$$
\{0\} \subset \operatorname{Image} L_\lambda^{n-1} \subset \operatorname{Image} L_\lambda^{n-2} \subset \cdots \subset \operatorname{Image} L_\lambda \subset T \mathsf M^n,
$$ more precisely such that $\operatorname{Ker}L_{\lambda}^{n-k} = \operatorname{Image}L_{\lambda}^{k}= \operatorname{span}(\pd{}{u^{k+1}},\hdots,\pd{}{u^n})$.
    
\end{Remark}
\begin{Lemma} 
Let $\mathcal{F}_{k}$ be the foliation generated by $\operatorname{Image}L_{\lambda}^{k}$. Then any vector field e satisfying  $\mathcal{L}_{e}(L) = \operatorname{Id}$ is $\mathcal{F}_{k}$-preserving, i.e., for any vector field $\Tilde{\eta} \in \operatorname{Image}L_{\lambda}^{k}$ holds that $\left[e,\Tilde{\eta}\right] \in \operatorname{Image}L_{\lambda}^{k} $.
\end{Lemma}
\begin{proof}
Let $\Tilde{\eta} = L_{\lambda}^{k}\eta$. Then we have:
\begin{equation*}
\left[e, L_{\lambda}^{k}\eta\right] = \mathcal{L}_{e}(L_{\lambda}^{k})\eta + L_{\lambda}^{k}\left[e, \eta\right] = kL_{\lambda}^{k-1}\mathcal{L}_{e}(L_{\lambda})\eta + L_{\lambda}^{k}\left[e, \eta\right]. 
\end{equation*}
Note that $\mathcal{L}_{e}(L) = \operatorname{Id}$. Taking into account that  $\mathcal{L}_{e}(\lambda)= 1$ we get $\mathcal{L}_{e}(L_{\lambda}) = 0$. Finally:
\begin{equation*}
\left[e, L_{\lambda}^{k}\eta\right] =  L_{\lambda}^{k}\left[e, \eta\right].   
\end{equation*}
Therefore, $\left[e,\Tilde{\eta}\right] \in \operatorname{Image}L_{\lambda}^{k} $.
\end{proof}
We start with the coordinate system $(u^1,\hdots, u^n)$ from Theorem \ref{thmjrdn}. In this coordinate system  the vector field $e =e^{i}\pd{}{u^{i}}$ where $e^{i} = e^{i}(u^{1}, \hdots, u^{i})$. Note that the eigenvalue of $L$ is $\lambda = u^1$ and $ e^1 = \mathcal{L}_{e}(u^1)= 1$.
\vskip 5pt
Let us denote the dual operator of $L$ by $L^*$. Then for $2 \leq k \leq n$
\begin{align*}
  &du^{k} =  \mathcal{L}_{e}(L^{*})du^{k} = \mathcal{L}_{e}(L^{*}du^{k}) - L^{*}\mathcal{L}_{e}(du^{k}) = \mathcal{L}_{e}(u^{1}du^{k}+du^{k-1}) - L^{*}de^{k} = \\
  &du^{k} + u^{1}de^{k} + de^{k-1} - L^{*}de^{k}.
\end{align*}
Thus we get
\begin{equation*}
    (L^{*}-u^1\operatorname{Id})de^{k} = de^{k-1},
\end{equation*}
or, equivalently,
\begin{equation*}
    f(u)du^{1} + \sum_{j=2}^{k-1} \frac{\partial e^{k}}{\partial u^{j+1}}du^{j} = \sum_{j=1}^{k-1} \frac{\partial e^{k-1}}{\partial u^{j}}du^{j}, \quad \text{where} \quad f(u) = \frac{\partial e^{k}}{\partial u^{2}} - \sum_{j=3}^{k} (j-2)u^{j}\frac{\partial e^{k}}{\partial u^{j}}
\end{equation*}
Together with $e^{1} = 1$ we have for $2 \leq k \leq n$
\begin{align*}
    &\frac{\partial e^{k}}{\partial u^j} =  \frac{\partial e^{k-1}}{\partial u^{j-1}}  \quad \text{for} \quad 3 \leq j \leq k -1,  \\
    &\frac{\partial e^{k}}{\partial u^{2}} - \sum_{j=3}^{k} (j-2)u^{j}\frac{\partial e^{k}}{\partial u^{j}} = \frac{\partial e^{k-1}}{\partial u^{1}}, \\
    & \frac{\partial e^{k}}{\partial u^{k}} = 0.
\end{align*}
\begin{Corollary}
\label{cor1}
    If $e = \pd{}{u^{1}} + e^{n}\pd{}{u^{n}} $, then $e^{n} = e^{n}(u^1)$, i.e., $e^{n}$ depends only on the first coordinate $u^{1}$.
\end{Corollary}
Let $L_k$ denote the $k \times k$ submatrix of $L$ composed by $L_{j}^{i}$ with $1 \leq i,j \leq k$ and $e_{k} =  \sum_{j=1}^{k} e^{j} \pd{}{u^{j}}$. We start with the pair $(L_1, e_1)$, where submatrix $L_1=(u^1)$ and $e_1 = \pd{}{u^{1}}$. The pair $(L_1, e_1)$ already has a canonical form, so we set $y^1=u^1$. Then we reduce $(L_2, e_2)$ to the 2-dimensional canonical form $(L, \pd{}{y^{1}})$ by changing only one coordinate $u^2 \mapsto y^2$ and leaving all the others unchanged. And so on, assuming that $(L_k, e_k) =(L_k, \pd{}{y^{1}}) $ (which means that the first canonical coordinates $y^1, \ldots, y^k$ have been already constructed), we reduce $(L_{k+1}, e_{k+1})$ to the canonical form $(L_{k+1}, \pd{}{y^{1}})$ by finding the next canonical coordinate $y^{k+1}$ in terms of $y^1, \ldots, y^k$ and $u^{k+1}$. The process finishes in $n-1$ steps. To simplify the notations we prove the induction step of this procedure in case $k=n-1$.

\begin{Lemma}
 Suppose that in a coordinate system $(y^1, \hdots, y^{n-1},u^n)$ the Nijenhuis operator $L$ is in its canonical form as in Theorem \ref{thmjrdn} and $e=\pd{}{y^1} + e^{n}(y^1)\pd{}{u^n}$. Then we may change the last coordinate $u^n \mapsto y^n$ in such a way that $L$ remains in the canonical form and $e = \pd{}{y^1}$.
\end{Lemma}
\begin{proof}
 Let $y^n := u^n - \int_{0}^{y^1} e^{n}(t) dt $. Then we obtain 
 \begin{equation*}
     e(y^n) = e(u^n) - e\left(\int_{0}^{y^1} e^{n}(t) dt\right) = e^{n} - e^{n} = 0.  
 \end{equation*}
Thus in a new coordinate system $(y^1, \hdots, y^{n})$ the vector field $e=\pd{}{y^1}$. It remains to show that the canonical form of $L$ is preserved, i.e., the following relation is satisfied:
 \begin{equation*}
     L^{*}dy^{n} = -(n-2)y^{n}dy^{1} + dy^{n-1} + y^{1}dy^{n}.
 \end{equation*}
The following straightforward computation finishes the proof:
\begin{equation*}
    L^{*}dy^{n} = L^{*}du^{n} - L^{*}d\left(\int_{0}^{y^1} e^{n}(t) dt\right) = L^{*}du^{n} - y^{1}e^{n}(y^{1})dy^{1}, 
\end{equation*}
\begin{equation*}
dy^{n} = du^{n} - d\left(\int_{0}^{y^1} e^{n}(t) dt\right) = du^{n} - e^{n}(y^{1})dy^{1},
\end{equation*}
And, finally,
\begin{equation*}
 L^{*}dy^{n} = -(n-2)y^{n}dy^{1} + dy^{n-1} + y^{1}dy^{n}.   
\end{equation*}
\end{proof}
\begin{Corollary}
    \label{jordanblockframe}
Suppose that in a neighborhood of a generic point $\mathsf p \in \mathsf M^n$ a Nijenhuis operator $L$ with a unity $e$ is similar to the standard Jordan block with a non-constant real eigenvalue $\lambda$. Then vector fields $\left\{X_0, X_1, \dots, X_{n-1}\right\}$, where $X_k:=L^{k}e$, form a frame in a neighbourhood of $\mathsf p$. 
\end{Corollary}
It is left to prove the uniqueness of the coordinate system.

Any automorphism $\psi: U(p) \rightarrow U(p)$ preserving $\left(L, e\right)$ has to preserve any vector field $X$, i.e.,
\begin{equation*}
    \psi_{*} X = X.
\end{equation*}
Equivalently,
\begin{equation}
    \psi \circ \phi_{X}^{t} = \phi_{X}^{t} \circ \psi, \quad \text{where $\phi_{X}$ is the flow generated by $X$.}
\end{equation}
Hence, $\psi = \operatorname{id}$.
\subsection{Case of pair complex conjugate eigenvalues \texorpdfstring{$\mu, \Bar{\mu}$}{m,m}}
In the case of complex conjugate eigenvalues, we may introduce a canonical complex structure $J$ (see \cite{bkm1}, section 3.3), with respect to which $L$ is a complex Nijenhuis operator. Then we bring $(L,e)$ to its canonical form in the same manner as we did in the real case. 

\section{Proof of Theorem \ref{t3}}
\begin{Remark} \label{r1}
If a pair $\left(L, e\right)$ is a Nijenhuis operator with a unity on a manifold $\mathsf M$, then $(\Tilde{L}, e)$ is also a Nijenhuis operator with a unity, where $\Tilde{L}: = L - c \cdot\operatorname{Id}$ and $c$ is a constant.
\end{Remark}
We may assume by Remark \ref{r1} that $\operatorname{tr} L(\mathsf p) = 0$. In two-dimensional case the characteristic polynomial $\chi_{L}(t)$ is 
\begin{equation*}
    \chi_{L}(t) = t^2 -\operatorname{tr} L + \operatorname{det} L.
\end{equation*}
By Corollary \ref{charcoefsrelationscor} we get the following differential conditions on $\operatorname{tr} L$ and $\operatorname{det} L$ :
\begin{align*}
    e(\operatorname{tr} L) &= 2,   &   e(\operatorname{det} L) = \operatorname{tr} L. 
\end{align*}
Thus we may find a local coordinate system $\left(x,y\right)$ such that $e = \pd{}{x}$ and $\operatorname{tr} L = 2x, \operatorname{det} L = x^2 + f(y)$, where $f(y)$ is an analytic function. We distinguish the following cases:
\begin{enumerate}
    \item Let $df \neq 0$ at $\mathsf 0$. In a new coordinate system $x_{\text{new}} = x, y_{\text{new}} = f(y) - d$ we get
    \begin{align*}
         \operatorname{tr} L &= 2x_{\text{new}},   &   \operatorname{det} L &= x^2_{\text{new}} + y_{\text{new}} + d.
    \end{align*}
    Now suppose that $L$ is as follows:
    \begin{equation*}
        L = \left(\begin{array}{cc}
           l^1_1  & l^1_2 \\
           l^2_1  & l^2_2
        \end{array}\right), \quad \text{where} \quad l^1_1 = x - a(y),\, l^2_2 = x + a(y),\, l^1_2 = l^1_2(y),\, l^2_1 = l^2_1(y).
    \end{equation*}
    A 2-dimensional operator field is Nijenhuis if and only if the following invariant equality is satisfied (see \cite{k}):
    \begin{equation}
    \label{nijenhuis2dim}
        L^{*} \dd \operatorname{det} L = \operatorname{det} L \cdot \dd \operatorname{tr} L
    \end{equation}
    this yields
    \begin{align*}
    2x l^1_2 = -x - a (y) \quad &\Longrightarrow \quad    a(y) \equiv 0, l^1_2 = -\frac{1}{2}, \\
    2x (x- a(y)) + l^2_1 = 2(x^2 +y +d) \quad &\Longrightarrow \quad l^2_1 = 2(y+d).
    \end{align*}
    \item Let $f(y) = \pm y^kg(y) + d$, where $k \in \mathbb{Z}_{\geq 2},\, g(0) > 0,\, d \in \mathbb{R}$. In a new coordinate system $x_{\text{new}} = x,\, y_{\text{new}} = yg(y)^\frac{1}{k}$ we have
    \begin{align*}
         \operatorname{tr} L &= 2x_{\text{new}},   &   \operatorname{det} L &= x^2_{\text{new}} \pm y^k_{new} + d.
    \end{align*}
    Using the notations and  \eqref{nijenhuis2dim} as in previous case, we obtain:
    \begin{align*}
    2x l^1_2 = \mp k y^{k-1} (x + a(y)) \quad &\Longrightarrow \quad    a(y) \equiv 0,\, l^1_2 = \mp \frac{k}{2} y^{k-1}, \\
    2x (x- a(y)) \pm k y^{k-1} l^2_1 = 2(x^2 \pm y^k +d) \quad &\Longrightarrow \quad d = 0,\, l^2_1 = \frac{2}{k} y.
    \end{align*}
    \item Let $f(y) \equiv d$, where $d \in \mathbb{R}$. Finally, \eqref{nijenhuis2dim} provides the following:
    \begin{align*}
    2x l^1_2 = 0 \quad &\Longrightarrow \quad     l^1_2 = 0, \\
    2x (x- a(y))  = 2(x^2 +d) \quad &\Longrightarrow \quad d = 0, a(y) \equiv 0.
    \end{align*}
\end{enumerate}
The obtained formulas in every considered case finish the proof.

\section{Proof of Theorem \ref{t4} and Corollaries \ref{c1}, \ref{c2}}

    By Remark \ref{r1} we may assume that $\operatorname{tr}(L)(\mathsf p) = 0$ at $\mathsf p \in \mathsf M^3$. Let $\chi_{L}(t)$ be the characteristic polynomial of $L$ defined as in Corollary \ref{charcoefsrelationscor}. We have:
    \begin{equation*}
        e(\sigma_{1}) = -3,  \qquad e(\sigma_{2}) = - 2\sigma_{1},  \qquad e(\sigma_{3}) = -\sigma_{2}. 
    \end{equation*}
Therefore, we may introduce a local coordinate system $(x^1,x^2,x^3)$ such that
\begin{align*}
&e = \pd{}{x^1}, \quad \sigma_1 = -3x^1, \quad \sigma_2 = 3(x^1)^2 + f(x^2,x^3), \quad \sigma_3 = -(x^{1})^3 - x^1 f(x^2,x^3) + g(x^2,x^3), 
\end{align*}
where $f,g$ are some functions.

The characteristic polynomial $\chi_{L}(t)$ is as follows:
\begin{equation*}
 \chi_{L}(t)=t^3 -3 x^{1} t^{2}+(3(x^{1})^{2} + f(x^2,x^3)) t -(x^{1})^{3} - x^1 f(x^2,x^3) + g(x^2,x^3).
\end{equation*}
Since $L(\mathsf p)$ has a multiple eigenvalue at every point $\mathsf p \in \mathsf M^3$, the discriminant $D$ of $\chi_{L}(t)$ is equal to zero, namely:
\begin{equation*}
    D = -4f^3 - 27g^2 = 0 \quad \iff f = - \frac{3}{4^{\frac{1}{3}}}g^{\frac{2}{3}}.
\end{equation*}
\begin{Lemma}
\label{3dimfirstlemma}
    Let $f, g$ be analytic functions of two variables such that $f = - g^{\frac{2}{3}}$. Then there exists an analytic function $h$ satisfying $ h^2 = -f, h^3 = g$.
\end{Lemma}
\begin{proof}
Let a function $h$ be defined as $h: = g^{\frac{1}{3}}$. Then we have
\begin{equation*}
    h = \frac{g}{g^{\frac{2}{3}}} = - \frac{g}{f}. 
\end{equation*}
Thus $h$ is meromorphic. But at the same time $h^2 = -f$, so $h$ is analytic.
\end{proof}
By Lemma \ref{3dimfirstlemma} we get:
\begin{align*}
    &f(x^2,x^3) = - h^{2}(x^2,x^3), \\
    &g(x^2,x^3) = \frac{2\sqrt{3}}{9}h^{3}(x^2,x^3).
\end{align*}
Using these equalities we obtain:
\begin{align*}
    &\chi_{L}(t) = t^3 -3 x^{1} t^{2}+(3(x^{1})^{2} - h^{2}) t -(x^{1})^{3} + x^1 h^{2} + \frac{2\sqrt{3}}{9}h^{3} = \\
    & (t - (x^1 + \frac{\sqrt{3}}{3}h))^2(t - (x^1 - \frac{2\sqrt{3}}{3}h)),\\
    & \lambda_{1} = x^1-\frac{2\sqrt{3}}{3}h, \quad \text{where $\lambda_{1}$ is eigenfunction of $L$ of multiplicity 1},\\
    & \lambda_{2} = x^{1}+\frac{\sqrt{3}}{3}h, \quad \text{where $\lambda_{2}$ is eigenfunction of $L$ of multiplicity 2.}
\end{align*}
Now we take a new coordinate system $(x^1,x^2,x^3)$ in which $\lambda_{1} = x^1, e = \pd{}{x^1}$. Let $L$ be of the following form:
\begin{equation*}
  L = \left(\begin{array}{ccc}
      x^1 & 0 & 0 \\
       0 & x^1 & 0 \\
        0 & 0 & x^1
  \end{array}\right) + \left(\begin{array}{ccc}
      l^{1}_{1} &  l^{1}_{2} &  l^{1}_{3} \\
        l^{2}_{1} &  l^{2}_{2} &  l^{2}_{3} \\
         l^{3}_{1} &  l^{3}_{2} &  l^{3}_{3}
  \end{array}\right), \quad \text{where $ l^{i}_{j}= l^{i}_{j}(x^2,x^3)$}.
\end{equation*}
Since $L$ is a Nijenhuis operator, it satisfies the following invariant equation:
\begin{equation}
\label{inveq}
    (L - \lambda\operatorname{Id})^{*}\operatorname{d}\lambda = 0, \quad \text{where $\lambda$ is an eigenfunction of $L$.}
\end{equation}
In the coordinate system $(x^1,x^2,x^3)$ it implies that $l^{1}_{1} = 0,\, l^{1}_{2} = 0,\, l^{1}_{3} = 0$. Thus we have
\begin{equation*}
L = \left(\begin{array}{ccc}
      x^1 & 0 & 0 \\
       0 & x^1 & 0 \\
        0 & 0 & x^1
  \end{array}\right) + \left(\begin{array}{ccc}
      0 & 0 & 0 \\
       l_{21} & l_{22} & l_{23} \\
        l_{31} & l_{32} & l_{33}
  \end{array}\right).
\end{equation*}
Note that since $\mathcal{F}:=\{x^{1} = const\}$ is an $L$-invariant foliation, a restriction of $L$ to the leaf $\{x^1 = 0\}$, i.e., $L' := L|_{\{x^1=0\}} = \left(\begin{array}{cc}
   l_{22}  & l_{23} \\
    l_{32} & l_{33}
\end{array} \right)$, is a Nijenhuis operator with an eigenfunction $\lambda' = \lambda'(x_2,x_3)$ of multiplicity 2.
\begin{Lemma}
    \label{3dimsecondlemma}
    Let $L$ be a Nijenhuis operator on a 2-dimensional manifold $\mathsf M^2$. Additionally, assume that $L(\mathsf p)$ has a multiple eigenvalue $\lambda$ at every point $\mathsf p \in \mathsf M^2$. Then in a neighbourhood of a point $\mathsf p_0 \in \mathsf M^2$ there exists a local coordinate system $(x,y)$ such that $L$ takes the following form:
    \begin{equation*}
    L =
     \left(\begin{array}{cc}
      \pm x^k & 0  \\
       g(x,y) & \pm x^k
  \end{array}\right).   
    \end{equation*}
\end{Lemma}
\begin{proof}
    Let $L$ be given by a matrix in some local coordinate system $(x,y)$:
    \begin{equation*}
    L =  \left(\begin{array}{cc}
      l^1_1 & l^1_2  \\
       l^2_1 & l^2_1
  \end{array}\right).  
       \end{equation*}
       We introduce the functions
       \begin{equation*}
           a(x,y) = \frac{1}{2}(l^1_1 - l^2_2), \quad f(x,y) = \frac{2l^1_2}{l^1_1 - l^2_2}, \quad g(x,y) = \frac{2 l^2_1}{l^1_1 - l^2_2}.
       \end{equation*}
       By construction we have $l^1_1 + l^2_2 = 2 \lambda$ and $l^1_1 l^2_2 - l^2_1 l^1_2 = \lambda^2$. This automatically implies that
       \begin{equation}
       \label{3dimlemmaeq}
           -a^2=l^2_1 l^1_2.
       \end{equation}
       Firstly, assume that $a$ is identically zero. This implies that at least one of $l_1^2$ and $l_2^1$ is identically zero and we deal with a triangular matrix with $\lambda$ on the diagonal. Thus the invariant formula \eqref{inveq} yields that the diagonal element, or more precisely the eigenfunction $\lambda$ depends only on one variable.

       Now assume that $a$ is not identically zero. We notice that function $a$ is analytic and $f, g$ are meromorphic. At least one of these functions is $\neq \infty$ at the coordinate origin. Indeed, if both approach infinity, then meromorphic function $f g$ approaches infinity as well. At the same time, $f g=-1$ almost everywhere by \eqref{3dimlemmaeq}. W.l.o.g. assume that this is $f$. By direct computation from \eqref{inveq} we get
$$
\frac{\partial \lambda}{\partial x}+f \frac{\partial \lambda}{\partial y}=0.
$$ 
Taking $\xi=(1, f)$ to be a coordinate vector field $\pd{}{y}$ we again arrive to the case where $\lambda$ depends on one variable and $L$ is as follows:
 \begin{equation*}
    L =  \left(\begin{array}{cc}
      \lambda(x) & 0 \\
       g(x,y) & \lambda(x)
  \end{array}\right)  
       \end{equation*}
We represent $\lambda(x) = \pm x^k g(x)$ where $ g(0)>0,\, k \in \mathbb{Z}_{>0}$. Then, defining a new coordinate system $x_{\text{new}} = x g^{\frac{1}{k}}(x),\, y_{\text{new}} = y$, $L$ takes the following form:
 \begin{equation*}
    L =  \left(\begin{array}{cc}
      \pm x^k & 0 \\
       g(x,y) & \pm x^k
  \end{array}\right).  
       \end{equation*}
\end{proof}
By Lemma \ref{3dimsecondlemma} $L$ is of the following form:
\begin{equation*}
    L = \left(\begin{array}{ccc}
      x^1 & 0 & 0 \\
       h(x^2,x^3) & \pm (x^{2})^{k} + x^1& 0 \\
        f(x^2,x^3) & g(x^2,x^3) & \pm (x^{2})^{k} + x^1
  \end{array}\right),
\end{equation*}
where, $h,f,g$ are analytic functions. The invariant equation \eqref{inveq} for $\lambda = x^1 \pm x_2^k$ immediately implies that $h = \frac{x^2}{k}$. Hence, $L$ and $e$ are as follows:
\begin{equation}
\label{seminormalform}
 L=\left(\begin{array}{lll}
x^{1} & 0 & 0 \\
\frac{x^{2}}{k} &  \pm (x^{2})^{k} + x^1 & 0 \\
f(x^2,x^3) & g(x^2,x^3) & \pm (x^{2})^{k} + x^1 \\
\end{array}\right) ,\, e = \pd{}{x^1},
\end{equation}
and vanishing of Nijenhuis torsion ($\mathcal{N}_{L} \equiv 0$) implies that the following PDE has to be fulfilled:
\begin{equation}
\label{3dimnijenhuiscond}
 \frac{x^{2}}{k}\pd{g}{x^2} + f\pd{g}{x^3} - g\pd{f}{x^3} = \frac{k-1}{k} g. 
\end{equation}
\begin{Remark}
The coordinates in \eqref{seminormalform} are defined up to the coordinate transformation $y^1=x^1, y^2=x^2, y^3=h\left(x^2, x^3\right)$ where $\frac{\partial h}{\partial x^3}(0,0) \neq 0, h(0,0) =0$. Under this transformation, the functional parameters of \eqref{seminormalform} are transformed in two steps.
\begin{enumerate}
    \item  Functions are transformed according to formulas:

\begin{align}
\label{3dimcoordtransform1}
& g_{\textnormal{new}}\left(x^2, x^3\right)=g\left(x^2, x^3\right) \frac{\partial h}{\partial x^3}\left(x^2, x^3\right), \\
\label{3dimcoordtransform2}
& f_{\textnormal{new}}\left(x^2, x^3\right)=\frac{x^2}{k} \frac{\partial h}{\partial x^2}\left(x^2, x^3\right)+f\left(x^2, x^3\right) \frac{\partial h}{\partial x^3}\left(x^2, x^3\right) .
\end{align}

\item Substitute the inverse coordinate change, that is $x^1\left(y^1\right), x^2\left(y^2\right), y^3\left(x^2, x^3\right)$ into the functions to obtain $f_{\textnormal{new }}\left(y^2, y^3\right)$ and $g_{\textnormal{new}}=g\left(y^2, y^3\right)$.
\end{enumerate}
\end{Remark}
 We distinguish the following possible cases:
\subsection{\texorpdfstring{$f(0,0) \neq 0$}{f(0,0) not = 0}}
We propose that there exists a coordinate transformation $h(x^2,x^3)$ such that $f_{\text{new}} \equiv 1$. Indeed, by \eqref{3dimcoordtransform2} we have:
\begin{equation*}
    \pd{h}{x^3} = \frac{1 - \frac{x^2}{k}\pd{h}{x^2}}{f}.
\end{equation*}
Due to the Cauchy-Kovalevskaya theorem a solution of this equation exists and is uniquely defined by an initial condition $h(x^2,0)$. Note that $\pd{h}{x^3}(0,0) = \frac{1}{f(0,0)} \neq 0$, hence $h$ defines an appropriate coordinate transformation $y^1=x^1, y^2=x^2, y^3=h\left(x^2, x^3\right)$. 

In the coordinates $(y^1,y^2,y^3)$ the differential condition \eqref{3dimnijenhuiscond} is
\begin{equation*}
 \frac{y^{2}}{k}\pd{g}{y^2} + \pd{g}{y^3} =  \frac{k-1}{k} g.    
\end{equation*}
Hence, $g(y^2,y^3) = F(y^{2}e^{-\frac{y^{3}}{k}})e^{\frac{(k-1)y^{3}}{k}}$, where $F$ is an analytic function. 

It is easy to check that transformations preserving $f_{\text{new}} \equiv 1$ are of the form:
\begin{equation*}
    z^1 = y^1, z^2 = y^2, z^3 = y^3 + z^2 q(z^2), \quad \text{where $q$ is an analytic function,}
\end{equation*}
and the functional parameter $g$ respects its transformation law \eqref{3dimcoordtransform1}, i.e.,  
\begin{equation*}
    g_{\text{new}}(z^2,z^3) = g(z^2, z^3 - z^2 q(z^2)).
\end{equation*}

\subsection{\texorpdfstring{$g(0,0) \neq 0, f(0,0) = 0$}{g(0,0) not = 0, f(0,0) = 0}}
Similarly, we are looking for a transformation $h(x^2,x^3)$ such that $g_{\text{new}} \equiv 1$. From \eqref{3dimcoordtransform1} we obtain:
\begin{equation*}
  \pd{h}{x^3} = \frac{1}{g}.  
\end{equation*}
The Cauchy-Kovalevskaya theorem guarantees that a solution to this equation exists and is uniquely defined by an initial condition $h(x^2,0)$. Note that $\pd{h}{x^3}(0,0) = \frac{1}{g(0,0)} \neq 0$, hence $h$ defines an appropriate coordinate transformation $y^1=x^1, y^2=x^2, y^3=h\left(x^2, x^3\right)$. 
In the coordinates $(y^1,y^2,y^3)$ the differential condition \eqref{3dimnijenhuiscond} is
\begin{equation*}
 \pd{f}{y^3}  = - \frac{k-1}{k}.     
\end{equation*}
Therefore, $f(y^2,y^3) =- \frac{k-1}{k}y^3 + y^2 s(y^2)$, where $s$ is an arbitrary function. The shape of the second summand follows from the assumption $f(0,0) = 0$.

As in the previous case, transformations preserving $g_{\text{new}} \equiv 1$ are as follows:
\begin{equation*}
z^1 = y^1, z^2 = y^2, z^3 = y^3 + y^2 q(y^2), \quad \text{where $q$ is an analytic function}
\end{equation*}
By the transformation law \eqref{3dimcoordtransform2} of $f$ we get:
\begin{equation*}
 f_{\text{new}} = -\frac{k-1}{k}z^3 + (\frac{z^2}{k}(q +z^2 q') + z^2 s + \frac{k-1}{k} z^2 q).  
\end{equation*}
 Consider the following ODE:
 \begin{equation}
 \label{3dimode}
     \frac{z^2}{k}q' + q = -s.
 \end{equation}
Let $s$ be given by a power series, i.e., $s =\sum_{i=0}^{\infty} a_{i}(z^{2})^i $. Then a power series of $q$ is uniquely defined by $s$ and is convergent. More precisely, if $q = \sum_{i=0}^{\infty} b_{i}(z^{2})^i$,  we have:
\begin{equation*}
    b_i = -\frac{k}{k+i} a_i, \quad \text{for $i \in \mathbb{Z}_{\geq 0}$}.
\end{equation*}
Hence, there exists a unique local coordinate system $(z^1, z^2, z^3)$ that brings $L$ to the desired form.

\section{Proof of Theorem \ref{t5}}
We will first demonstrate that the class of regular $F$-manifolds belongs to the class of Nijenhuis manifolds with a cyclic unity.

Let $\left(\mathsf M^{n}, \circ, e, E\right)$ be a regular $F$-manifold. Then the pair $\left(L, e\right)$, where $L:=E\circ$, is a Nijenhuis operator with a unity. Indeed, the conditions \eqref{hertlingmanincond} and \eqref{eulervectorfieldcond}  imply
\begin{equation*}
   \mathcal{L}_{e}(L) = \mathcal{L}_{e}(E\circ)= \left[e, E\right]\circ + E\mathcal{L}_{e}(\circ)= e\, \circ = \operatorname{Id}. 
\end{equation*}
Thus, the triple $(\mathsf M^{n}, L, e)$ form a Nijenhuis manifold with a unity. It remains to demonstrate that the unity $e$ is cyclic for $L$.

 Due to $\gl$-regularity of $L$, there exists a cyclic vector field $\xi$. Suppose that there exists an $n$-tuple $\left(\lambda_0, \dots, \lambda_{n-1}\right)$, such that
   \begin{equation*}
       \lambda_{0}e + \hdots + \lambda_{n-1}L^{n-1}e = 0.
   \end{equation*}
   Multiplying both sides of this equality by $\xi$ we obtain:
   \begin{equation*}
    \lambda_{0}\xi + \hdots + \lambda_{n-1}L^{n-1}\xi = 0.   
   \end{equation*}
   As $\xi$ is cyclic, we get
   \begin{equation*}
       \left(\lambda_0, \dots, \lambda_{n-1}\right) = \left(0,\dots,0\right).
   \end{equation*}
   Therefore, $e$ is a cyclic vector field.

Now, let us show that a Nijenhuis manifold with a cyclic unity $\left(\mathsf M^n, L, e\right)$ possesses a natural structure of an $F$-manifold.

We define a $(1,2)$ - tensor field $\circ$ on $\mathsf M^n$ as follows:
\begin{equation*}
   X_i \circ X_j := X_{i+j}, \quad \text{$\circ$ is $C^{\infty}$-linear in both arguments}, \quad \text{where $X_{i}:=L^{i}e$}.
\end{equation*}
Clearly, so defined multiplication $\circ$ is commutative, associative, and $X_0 = e$ is the unity.
\begin{Lemma}
\label{nijenhuiswithunityprop}
Let $L$ be a Nijenhuis operator on a manifold $\mathsf M^n$ and $e$ is a unity to $L$. Then the following holds:
\begin{align}
 \label{comrel1}
&\mathcal{L}_{X_{i}}(L^{j})= jL^{i+j-1}, \quad \text{where} \quad X_{i} = L^{i}(e), \quad  i,j \in \mathbb{Z}_{\geq 0}, \quad L^{0}:= \operatorname{Id}. \\
\label{comrel2}
&\left[X_{i}, X_{j}\right] = (j-i)X_{i+j-1} \quad \text{for} \quad i,j \in \mathbb{Z}_{\geq 0}.
\end{align}
\end{Lemma}
\begin{proof}
    Due to  \eqref{nijendef2} we have
    \begin{equation*}
        \mathcal{L}_{X_{i+1}}(L) = L\mathcal{L}_{X_{i}}(L),
    \end{equation*}
    and together with $\mathcal{L}_{e}(L) = \operatorname{Id}$ we get
      \begin{equation*}
        \mathcal{L}_{X_{i}}(L) = L^{i}.
    \end{equation*}
    Applying the Leibniz rule for a Lie derivative we obtain
    \begin{equation*}
       \mathcal{L}_{X_{i}}(L^{j}) =  jL^{j-1}\mathcal{L}_{X_{i}}(L) = jL^{i+j-1}.
    \end{equation*}
Without loss of generality, we can assume that $j > i$. Then the second formula can be derived as follows:
    \begin{equation*}
     \left[X_{i}, X_{j}\right] = \left[X_{i}, L^{j-i}X_{i}\right] = \mathcal{L}_{X_{i}}(L^{j-i})X_{i} + L^{j-i}\left[X_{i}, X_{i}\right]  = (j-i)L^{j-1}X_{i} = (j-i)X_{i+j-1}. 
    \end{equation*}
    \end{proof}
 We need to show that $(M, \circ, e, E)$ is an $F$-manifold, i.e., conditions (\ref{hertlingmanincond}) and (\ref{eulervectorfieldcond}) are satisfied. Note that it is sufficient to prove the conditions for the frame $\left\{X_0, X_1, \cdots, X_{n-1}\right\}$. Using the definition of $\circ$ and (\ref{comrel2}) we obtain
\begin{align}
    &\begin{aligned}
&[X_i \circ X_j, X_k \circ X_l] - [X_i \circ X_j, X_k] \circ X_l - X_k \circ [X_i \circ X_j, X_l] - X_i \circ [X_j, X_k \circ X_l] + \\
&X_i \circ [X_j, X_k] \circ X_l + X_i \circ X_k \circ [X_j, X_l] - X_j \circ [X_i, X_k \circ X_l] + X_j \circ [X_i, X_k] \circ X_l +  \\
 &X_j \circ X_k \circ [X_i, X_l] = [X_{i+j}, X_{k+l}] - [ X_{i+j}, X_k] \circ X_l - X_k \circ [X_{i+j}, X_l] - X_i \circ [X_j, X_{k+l}] +  \\
 &X_{i+l} \circ [X_j, X_k] + X_{i+k} \circ [X_j, X_l] - X_j \circ [X_i, X_{k+l}] + X_{j+l} \circ [X_i, X_k]  + X_{j+k} \circ [X_i, X_l] = \\
    & (k+l-i-j)X_{i+j+k+l-1} - (k-i-j)X_{i+j+k+l-1} - (l-i-j) X_{i+j+k+l-1} - \\
    & - (k+l-j) X_{i+j+k+l-1} + (k-j) X_{i+j+k+l-1} + (l-j) X_{i+j+k+l-1} - \\
    & - (k+l-i) X_{i+j+k+l-1} + (k-i) X_{i+j+k+l-1} + (l-i) X_{i+j+k+l-1} = 0,
    \end{aligned}
    \\
    &\begin{aligned}
 &[X_1, X_i \circ X_j] - [X_1, X_i]\circ X_j - X_i \circ [X_1, X_j] = \\
  &(i+j-1) X_{i+j} - (i-1) X_{i+j} - (j-1) X_{i+j} = X_i \circ X_j.
 \end{aligned}
\end{align}

\section{Proof of Theorem \ref{t6}}

We begin with the semi-normal form of $L$ described in Theorem \ref{t4}. Since the unity $e$ is a cyclic vector field for $L$ at a generic point, we may define a multiplication $\circ$ as in Theorem \ref{t5} and $\left(\mathsf M^3, \circ, e, E\right)$ will be an $F$-manifold.

Our goal is to obtain structural constants $c^{i}_{j k}$ of $\circ = c^{i}_{j k} \pd{}{x^i} \otimes dx^j \otimes dx^k$ in the local coordinates $\left(x^1,x^2,x^3\right)$. 
Since $e$ is the unity with respect to multiplication $\circ$, we have:
\begin{equation*}
    \pd{}{x^1} \circ \pd{}{x^i} = \pd{}{x^i}, \quad i = 1, 2, 3,
\end{equation*}
or, equivalently,
\begin{equation*}
    c^{i}_{j, 1} =
\begin{cases}
1, \quad \text{if $i = j$},\\
0, \quad \text{otherwise}.
\end{cases}
\end{equation*}
By a straightforward computation, we have:
\begin{equation*}
\begin{aligned}
     X_1 &= (x^1 + \lambda_0) \pd{}{x^1} + \frac{x^2}{k}\pd{}{x^2}+f\pd{}{x^3},  \\
    X_2 &= (x^1 + \lambda_0)^2 \pd{}{x^1} + \frac{x^2(\pm(x^2)^k+2x^1+2\lambda_0)}{k}\pd{}{x^2} + \frac{2(\lambda_0+x^1\pm\frac{(x^2)^k}{2})kf + gx^2}{k}\pd{}{x^3}, \\
     X_3 &= (x^1+\lambda_0)^3\pd{}{x^1} + \frac{x^2((x^2)^k + 3(x^1 + \lambda_0)(\pm(x^2)^k+x^1+\lambda_0))}{k}\pd{}{x^2} +  \\
&\frac{\pm2g(x^{2})^{k+1}+kf(x^{2})^{2k}+3((\pm(x^2)^k+x^1+\lambda_0)kf+gx^2)(x^1+\lambda_0)}{k} \pd{}{x^3}, \\
X_4 &= (x^1+\lambda_0)^4\pd{}{x^1} + \frac{x^2}{k}(4(x^2)^{2k}x^1+4(x^2)^{2k}\lambda_0\pm(x^2)^{3k}\pm \\
&6(x^2)^k\lambda^2_0\pm12(x^2)^kx^1\lambda_0\pm6(x^2)^k(x^1)^2+4\lambda^3_0+12x^1\lambda^2_0+12(x^1)^2\lambda_0+4(x^1)^3)\pd{}{x^2} + \\
&\frac{1}{k}(3g(x^2)^{2k+1} \pm8g(x^1+\lambda_0)(x^2)^{k+1}+4kf(x^1+\lambda_0)(x^2)^{2k} \pm kf(x^2)^{3k}+ \\
&4((\lambda_0 + x^1 \pm \frac{3(x^2)^k}{2})kf+\frac{3gx^2}{2})(x^1+\lambda_0)^2)\pd{}{x^3}.
\end{aligned}
\end{equation*}
And the following equalities hold:
\begin{equation*}
    \begin{aligned}
        X_1 \circ X_1 &= X_2, \\
        X_1 \circ X_2 &= X_3, \\
        X_2 \circ X_2 &= X_4.
    \end{aligned}
\end{equation*}
This leads to three $3 \times 3$ linear systems, which are non-degenerate at points, where $g \neq 0$. The solutions of these systems are the following:
\begin{equation*}
    \begin{aligned}
        c^{1}_{22} &= 0, \quad &c^{1}_{23} &= 0, \quad &c^{1}_{33} &= 0, \\
        c^{2}_{22} &=\pm k(x^{2})^{k-1}, \quad &c^{2}_{23} &= 0, \quad &c^{2}_{33} &= 0, \\
        c^{3}_{22} &= \frac{k}{x^2}(kf(x^2)^{k-1} \mp g), \quad &c^{3}_{23} &=\pm k(x^{2})^{k-1}, \quad &c^{2}_{33} &= 0.
    \end{aligned} 
\end{equation*}
Let  $h(x^2,x^3):=\frac{k}{x^2}(kf(x^2)^{k-1} \mp g)$. Then we obtain
\begin{equation*}
  \begin{aligned}
    & \pd{}{x^1} \circ \pd{}{x^i} = \pd{}{x^i}, \quad i = 1, 2, 3, \\
    & \pd{}{x^2} \circ \pd{}{x^2} = \pm k(x^2)^{k - 1} \pd{}{x^2} + h(x^2, x^3) \pd{}{x^3}, \\
    & \pd{}{x^2} \circ \pd{}{x^3} = \pm k(x^2)^{k - 1} \pd{}{x^3}, \\
    & \pd{}{x^3} \circ \pd{}{x^3} = 0.
\end{aligned}  
\end{equation*}
Substituting $g = \mp (\frac{hx^2}{k} - kf(x^2)^{k-1})$ into \eqref{3dimnijenhuiscond} we get:

   \begin{equation*}
    \frac{x^2}{k}\pd{h}{x^2} + f\pd{h}{x^3} -k(x^2)^{k-1}\pd{f}{x^2} - h\pd{f}{x^3} = \frac{k-2}{k}h.
\end{equation*}
The transformation rules \eqref{3dimcoordtransform1},\eqref{3dimcoordtransform2} for a pair of  functional parameters $g,f$ lead to the following transformation rules for a pair $h,f$:
\begin{align*}
& \Bar{h}(x^2, r(x^2,x^3))=k(x^2)^{k-1}\pd{r}{x^2}\left(x^2,x^3\right)   +  h\left(x^2, x^3\right) \frac{\partial r}{\partial x^3}\left(x^2, x^3\right), \\
& \Bar{f}(x^2, r(x^2,x^3))=\frac{x^2}{k} \frac{\partial r}{\partial x^2}\left(x^2, x^3\right)+f\left(x^2, x^3\right) \frac{\partial r}{\partial x^3}\left(x^2, x^3\right) . 
\end{align*}
\newline
\textbf{Acknowledgements:} The research of E. Antonov was supported by the DFG grant (MA 2565/7, project number 455806247). The authors would like to thank Vladimir Matveev and Alexey Bolsinov for fruitful and substantial discussions.

\end{document}